\documentclass[12pt]{amsart}
\usepackage{graphicx}
\usepackage[active]{srcltx}
\usepackage{amsthm,mathrsfs}
\usepackage{a4wide}
\usepackage[english]{babel}
\usepackage{amsfonts}
\usepackage{amsmath}
\usepackage{amssymb}
\usepackage{dsfont}
\newtheorem{lemma}{Lemma}
\newtheorem{theorem}{Theorem}
\newtheorem{corollary}{Corollary}

\usepackage{float}
\usepackage{url}

\newcommand {\E} {\mathbb{E}}
\newcommand {\V} {\mathbb{V}}

\newcommand {\p} {\mathbb{P}}

\newcommand {\N} {\mathbb{N}}

\newcommand {\R} {\mathbb{R}}

\newcommand {\ve} {\varepsilon}

\makeatletter\def\blfootnote{\xdef\@thefnmark{}\@footnotetext}\makeatother
\allowdisplaybreaks
\title[The LIL for trigonometric series with bounded gaps]{\bf On the law of the iterated logarithm for trigonometric series with bounded gaps II}

\author{Christoph Aistleitner} 
\address{Department of Mathematics, Graduate School of Science, Kobe University, Kobe 657-8501, Japan}
\email{aistleitner@math.tugraz.at}

\author{Katusi Fukuyama}
\address{Department of Mathematics, Graduate School of Science, Kobe University, Kobe 657-8501, Japan}
\email{fukuyama@math.kobe-u.ac.jp}

\thanks{The first author is supported by a Schr\"odinger scholarship of the Austrian Research
Foundation (FWF). The second author is supported by KAKENHI 24340017 and 24340020.}


\begin{document}

\begin{abstract}
It is well-known that for a quickly increasing sequence $(n_k)_{k \geq 1}$ the functions $(\cos 2 \pi n_k x)_{k \geq 1}$ show a behavior which is typical for sequences of independent random variables. If the growth condition on $(n_k)_{k \geq 1}$ is relaxed then this almost-independent behavior generally fails. Still, probabilistic constructions show that for \emph{some} very slowly increasing sequences $(n_k)_{k \geq 1}$ this almost-independence property is preserved. For example, there exists $(n_k)_{k \geq 1}$ having bounded gaps such that the normalized sums $\sum \cos 2 \pi n_k x$ satisfy the central limit theorem (CLT). However, due to a ``loss of mass'' phenomenon the variance in the CLT for a sequence with bounded gaps is always smaller than $1/2$. In the case of the law of the iterated logarithm (LIL) the situation is different; as we proved in an earlier paper, there exists $(n_k)_{k \geq 1}$ with bounded gaps such that
$$
\limsup_{N \to \infty} \frac{\left| \sum_{k=1}^N \cos 2 \pi n_k x \right|}{\sqrt{N \log \log N}} = \infty \qquad \textup{for almost all $x$.}
$$
In the present paper we prove a complementary results showing that any prescribed limsup-behavior in the LIL is possible for sequences with bounded gaps. More precisely, we show that for any real number $\Lambda \geq 0$ there exists a sequence of integers $(n_k)_{k \geq 1}$ satisfying $n_{k+1} - n_{k} \in \{1,2\}$ such that the limsup in the LIL equals $\Lambda$ for almost all $x$. Similar results are proved for sums $\sum f(n_k x)$ and for the discrepancy of $(\langle n_k x \rangle)_{k \geq 1}$.
\end{abstract}

\date{}
\maketitle

\section{Introduction and statement of results}

It is well-known that for any quickly growing sequence of positive integers $(n_k)_{k \geq 1}$ the systems $(\cos 2 \pi n_k x)_{k \geq 1}$ and $(\sin 2 \pi n_k x)_{k \geq 1}$ exhibit many properties which are typical for sequences of \emph{independent} random variables. This similarity includes the analogs of the Kolmogorov three-series theorem (Kolmogorov, 1924), the central limit theorem (Salem--Zygmund, 1947) 				e law of the iterated logarithm (Erd\H os--G\'al, 1955), all of which hold if $(n_k)_{k \geq 1}$ satisfies the Hadamard gap condition
\begin{equation} \label{hadamard}
\frac{n_{k+1}}{n_k} \geq q > 1, \qquad k \geq 1.
\end{equation}
Similar results hold if the simple functions $\cos 2 \pi \cdot$ and $\sin 2 \pi \cdot$ are replaced by more general 1-periodic functions, satisfying some regularity conditions. Philipp~\cite{plt} proved that even an analog of the Chung--Smirnov law of the iterated logarithm holds: for any sequence $(n_k)_{k \geq 1}$ satisfying~\eqref{hadamard} we have
$$
\frac{1}{4} \leq \limsup_{N \to \infty} \frac{N D_N^*(\langle n_1 x \rangle, \dots, \langle n_N x\rangle)}{\sqrt{N \log \log N}} \leq C_q \qquad \textup{for almost all $x$.}
$$
Here $\langle \cdot \rangle$ stands for the fractional part of a real number, and 
$$
D_N^*(y_1, \dots, y_N) = \sup_{0 \leq a \leq 1} \left| \frac{1}{N} \sum_{k=1}^N\mathds{1}_{[0,a]} (y_k) - a \right|
$$
denotes the so-called \emph{star-discrepancy} of a set $y_1, \dots, y_N$ of points from the unit interval. In probabilistic terminology, the star-discrepancy is a version of the Kolmogorov--Smirnov statistic, adjusted to the uniform distribution on $[0,1]$, and usually applied to deterministic sequences. The notion of the star-discrepancy is closely related to the theory of uniform distribution modulo one, which is a branch of number theory that originated in the work of Borel, Weyl and others in the early 20th century. Classical survey papers on lacunary trigonometric series and their almost-independent behavior are for example~\cite{gapo,kac,kahane}; more recent survey papers are~\cite{ab,berkes2}. An introduction to uniform distribution modulo one and discrepancy theory can be found for example in the monographs~\cite{dts,knu}.\\

In some cases the gap condition~\eqref{hadamard} may be slightly weakened, but in general the almost-independent behavior of $(f(n_k x))_{k \geq 1}$ for 1-periodic $f$ breaks down without a strong growth condition on $(n_k)_{k \geq 1}$. However, this is only one part of the truth: while the probabilistic limit theorems fail to hold for \emph{all} sequences $(n_k)_{k \geq 1}$ without strong growth conditions, these limit theorems still remain true for \emph{some} slowly growing sequences $(n_k)_{k \geq 1}$, or actually for all ``typical'' slowly growing sequences in a suitable probabilistic model. Of fundamental importance is a paper of Salem and Zygmund~\cite{salemzyg}. Amongst other things, they proved the following: Let $(\xi_k)_{k \geq 1}$ be a sequence of independent, fair $\{-1,1\}$-valued random variables. Then the systems $(\xi_k \cos 2 \pi k x)_{k \geq 1}$ and $(\xi_k \sin 2 \pi k x)_{k \geq 1}$ satisfy the CLT and LIL for \emph{almost all} realizations of $(\xi_k)_{k \geq 1}$. Since for almost all $x$ we have
\begin{equation} \label{sumbounded}
\sum_{k=1}^N \cos 2\pi k x = \mathcal{O}(1) \qquad \textrm{and} \qquad \sum_{k=1}^N \sin 2\pi k x = \mathcal{O}(1) \qquad \textrm{as $N \to \infty$,}
\end{equation}
the same conclusion holds if we assume that the $\xi_k$'s are $\{0,1\}$-valued instead of $\{-1,1\}$-valued. Then we can define a (random) sequence $(n_k)_{k \geq 1}$ as the sequence containing all the numbers $\{k \geq 1:~\xi_k = 1\}$, and conclude that the CLT and LIL for $(\cos 2 \pi n_k x)_{k \geq 1}$ and $(\sin 2 \pi n_k x)_{k \geq 1}$ hold for almost all sequences $(n_k)_{k \geq 1}$.\\

Note that a sequence $(n_k)_{k \geq 1}$ constructed in this randomized way typically grows very slowly; by the strong law of large numbers we have $n_k / k \to 2$ almost surely, which means linear growth (in contrast to~\eqref{hadamard}, which implies exponential growth). However, linear growth does not necessarily imply that the gaps $n_{k+1} - n_k$ are small for all $k$. Actually, by the Erd\H os-R\'enyi ``pure heads'' theorem a random sequence $(n_k)_{k \geq 1}$ constructed in the described manner has infinitely many gaps $n_{k+1}-n_k$ of size roughly $\log k$, almost surely. Berkes~\cite{berkes} proved that slower growth is possible: for any function $h(k) \to \infty$ there exists a sequence $(n_k)_{k \geq 1}$ for which $1 \leq n_{k+1} - n_k \leq h(k)$ such that $(\cos 2 \pi n_k x)_{k \geq 1}$ satisfies the CLT. In a sense, Berkes' theorem is optimal: by a result of Bobkov and G\"otze~\cite{bobkov}, for $(n_k)_{k \geq 1}$ having bounded gaps, that is, satisfying $1 \leq n_{k+1} - n_k \leq K$ for some constant $K$, the CLT for $(\cos 2 \pi n_k x)_{k \geq 1}$ may hold, but due to a ``loss of mass'' phenomenon only with a limiting variance smaller than 1/2 (which would be the ``correct'' variance, corresponding to the independent case). For any variance less than $1/2$ there actually exist appropriate sequences having bounded gaps and satisfying the CLT; see~\cite{fuku2,fuku3}.\\

One could suspect that a similar ``loss of mass'' phenomenon should also appear in case of the LIL; namely, that for any sequence $(n_k)_{k \geq 1}$ satisfying $1 \leq n_{k+1} - n_k \leq K$ for some $K$ the limsup in the LIL
$$
\limsup_{N \to \infty} \frac{\left| \sum_{k=1}^N \cos 2 \pi n_k x \right|}{\sqrt{2 N \log \log N}} 
$$
should be less than $1/\sqrt{2}$ for almost all $x$ (which would be the ``correct'' value for independent random variables). However, surprisingly, the contrary is true. In a previous paper we proved the existence of a sequence $(n_k)_{k \geq 1}$ satisfying the strongest possible ```bounded gap'' condition $n_{k+1} - n_k \in \{1,2\},~k \geq 1,$ such that
\begin{equation} \label{largelimsup}
\limsup_{N \to \infty} \frac{\left| \sum_{k=1}^N \cos 2 \pi n_k x \right|}{\sqrt{N \log \log N}} = \infty \qquad \textup{for almost all $x$.}
\end{equation}
Berkes asked\footnote{At the 8th World Congress on Probability and Statistics in Istanbul, July 2012.} whether \emph{any} prescribed value for the limsup in the LIL is possible for sequences with bounded gaps, and if such results can be generalized to a larger class of $1$-periodic functions $f$ and to the discrepancy $D_N^*$. In both cases the answer is affirmative, as we want to show in the present paper.\\

Throughout this paper, $\| \cdot \|$ denotes the $L^2(0,1)$ norm, and $\| \cdot \|_\infty$ denotes the $L^\infty$ norm. As already mentioned, $\langle \cdot \rangle$ denotes the fractional part of a real number. We will write $\mu$ for the Lebesgue measure. Furthermore, we will write $\mathds{1}_A$ for the indicator function of an interval $A \subset [0,1]$, and $\mathbf{I}_A$ for the indicator function of $A$, centered and extended with period one; that is,
\begin{equation} \label{ia}
\mathbf{I}_A(x) = \mathds{1}_A(\langle x \rangle) - \mu(A).
\end{equation}

\begin{theorem} \label{th1}
For any given real number $\Lambda \geq 0$ and any function $f$ satisfying
\begin{equation} \label{f}
f(x+1)=f(x), \qquad \int_0^1 f(x) ~dx = 0, \qquad \textup{Var}_{[0,1]} f < \infty
\end{equation}
there exists a sequence of positive integers $(n_k)_{k \geq 1}$ satisfying
$$
n_{k+1} - n_k \in \{1,2\}
$$
such that we have
\begin{equation} \label{th1concl}
\limsup_{N \to \infty} \frac{\left| \sum_{k=1}^N f(n_k x) \right|}{\sqrt{N \log \log N}} = \Lambda \|f\| \qquad \textrm{for almost all $x$}. 
\end{equation}
If $\mathcal{F}$ is a countable class of functions satisfying~\eqref{f}, then there is a sequence $(n_k)_{k \geq 1}$ such that~\eqref{th1concl} holds for \emph{all} functions $f$ in $\mathcal{F}$.
\end{theorem}
Note that Theorem~\ref{th1} can be in particular applied to the functions $\cos 2 \pi x$ and $\sin 2 \pi x$, and shows that any prescribed limsup behavior in the LIL is possible for these functions. In this sense Theorem~\ref{th1} is a complementary result to~\eqref{largelimsup}.\\

Theorem~\ref{th2} shows that any prescribed LIL behavior is also possible for the discrepancy $D_N^*$.

\begin{theorem} \label{th2}
For any given real number $\Lambda>0$ there exists a sequence of positive integers $(n_k)_{k \geq 1}$ satisfying
$$
n_{k+1} - n_k \in \{1,2\}, \qquad k \geq 1,
$$
such that we have
$$
\limsup_{N \to \infty} \frac{N D_N^*(\langle n_1 x\rangle, \dots, \langle n_N x\rangle)}{\sqrt{N \log \log N}} = \Lambda \qquad \textrm{for almost all $x$}. 
$$
\end{theorem}
It is easily seen that Theorem~\ref{th2} remains true if the star-discrepancy $D_N^*$ is replaced by the \emph{extremal discrepancy} $D_N$, in which the supremum is taken over all intervals $[a,b] \subset [0,1]$ instead of only all anchored intervals $[0,a] \subset [0,1]$. For the sake of shortness we have confined ourselves to giving proofs only for $D_N^*$.\\

We want to make some more remarks, before turning to the proofs of Theorem~\ref{th1} and Theorem~\ref{th2}. Both of them are proved using a probabilistic construction, which is a typical feature of results about slowly growing sequences satisfying probabilistic limit theorems. As far as we know, no explicit example of a (in some sense) slowly growing sequence $(n_k)_{k \geq 1}$ for which $(f(n_k x))_{k \geq 1}$ for some appropriate $f$ satisfies the CLT or LIL is known.\\

Our probabilistic construction is similar to the one introduced by Salem--Zygmund, in the sense that for any $k$ we decide independently with a certain probability whether $k$ should be contained in our sequence $(n_k)_{k \geq 1}$ or not (actually, in our proof we form blocks of finitely many integers and always either take or discard the full block, but the general philosophy is the same). An other randomized way of generating slowly growing sequences $(n_k)_{k \geq 1}$ which are supposed to satisfy some probabilistic limit theorems is to let $(n_k)_{k \geq 1}$ be generated by a random walk; results for the model can be found for example in~\cite{bobkov,schatte,weber}.\\

Theorem~\ref{th2} has an interesting deterministic, number-theoretic counterpart. It is well-known that when taking $n_k=k,~k \geq 1$, then the discrepancy of $(\langle n_k x \rangle)_{k \geq 1}$ tends to zero almost as fast as $N^{-1}$, and in particular much faster than the speed of convergence specified by the Chung--Smirnov LIL (see Lemma~\ref{lemmadiscr} below for details). This is due to the close connection of the discrepancy of such sequences and continued fractions expansions, which was first observed by Ostrowski around 1920. Thus for almost all $x$ we have
\begin{equation} \label{arnold}
\limsup_{N \to \infty} \frac{N D_N^*(\langle x\rangle, \langle 2x \rangle, \dots, \langle N x\rangle)}{\sqrt{N \log \log N}} = 0. 
\end{equation}
However, as Arnol$'$d~\cite{arnold} showed, for any given $\Lambda$ there exists a number $x$ (which can be explicitly constructed in terms of its continued fractions expansion) such that the limsup in~\eqref{arnold} is at least $\Lambda$. We do not know if there also exists a construction of a real number $x$ for which the value of the limsup in~\eqref{arnold} is \emph{precisely} $\Lambda$.\\

Koksma's inequality (see Lemma~\ref{koksinequ} below) implies that for the sequence $(n_k)_{k \geq 1}$ satisfying~\eqref{largelimsup} we have
$$
\limsup_{N \to \infty} \frac{N D_N^*(\langle n_1 x\rangle, \dots, \langle n_N x\rangle)}{\sqrt{N \log \log N}} = \infty \qquad \textrm{for almost all $x$}. 
$$
It is a natural question to ask what for a sequence $(n_k)_{k \geq 1}$ having bounded gaps the slowest possible order of decay of $D_N^*(\langle n_1 x\rangle, \dots, \langle n_N x\rangle)$ for almost all $x$ is, and whether it can be significantly slower that $\sqrt{\log \log N}/\sqrt{N}$. A closely related question is the following: for any strictly increasing sequence $(n_k)_{k \geq 1}$ (not necessarily having bounded gaps), what is the slowest possible order of decay of $D_N^*(\langle n_1 x\rangle, \dots, \langle n_N x\rangle)$ for almost all $x$ \emph{in terms of the largest frequency $n_N$}? It is known that for any strictly increasing $(n_k)_{k \geq 1}$ we have
$$
N D_N^*(\langle n_1 x\rangle, \dots, \langle n_N x\rangle) = \mathcal{O} \left( \sqrt{N} (\log N)^{3/2+\ve} \right)
$$
for almost all $x$ (Baker~\cite{baker}), and that there exists a sequence $(n_k)_{k \geq 1}$ for which
$$
\limsup_{N \to \infty} \frac{N D_N^*(\langle n_1 x\rangle, \dots, \langle n_N x\rangle)}{\sqrt{N \log N}} = \infty
$$
for almost all $x$ (Berkes--Philipp~\cite{beph}). However, the sequence $(n_k)_{k \geq 1}$ in the Berkes--Philipp theorem grows very quickly (almost as fast as $e^{\sqrt{k}}$), and it would be very interesting to know if a similarly large discrepancy (for almost all $x$) is also possible for slowly growing $(n_k)_{k \geq 1}$. These problems are related to Carleson's theorem on the almost everywhere convergence of Fourier series, and also have a relation to certain sums involving greatest common divisors (see~\cite{abs}).\\

Finally, we want to mention the relation between our results and some problems in analysis. A \emph{Littlewood polynomial} is a polynomial all of whose coefficients are either $1$ or $-1$; that is, it is of the form
$$
p(x) = \sum_{k=0}^N a_k x^k, \qquad \textrm{where} \qquad  a_k \in \{-1,1\},~0\leq k \leq N. 
$$
It is a classical problem of analysis to study the possible size of $p(z)$ with respect to $L^p$ norms on the complex unit circle, in terms of $N$ (see for example~\cite{erdelyi} for a survey, and~\cite{borwein} for such polynomials in the context of random trigonometric series). The proof of Theorem~\ref{th1}, together with~\eqref{sumbounded} and~\eqref{largelimsup}, yields the following result, which can be seen as a result on the real parts of a sequence of Littlewood polynomials.
\begin{corollary} \label{co1}
For any $\Lambda \in [0,\infty]$ there exists a sequence $(a_k)_{k \geq 1} \in \{-1,1\}^\N$ such that
$$
\limsup_{N \to \infty} \frac{\left|\sum_{k=1}^N a_k \cos 2 \pi k x \right|}{\sqrt{N \log \log N}} = \Lambda \qquad \textup{for almost all $x$.}
$$
\end{corollary}
A similar result holds if $\cos 2\pi \cdot$ is replaced by $\sin 2 \pi \cdot$. Similar to the aforementioned problems concerning the discrepancy of $(\langle n_k x\rangle)_{k \geq 1}$ (and actually closely related to those problems), it would be interesting to know what the largest possible order of growth of
$$
\left|\sum_{k=1}^N a_k \cos 2 \pi k x \right|, \qquad a_k \in \{-1,1\},~k \geq 1,
$$
for almost all $x$ is. By Carleson's theorem this sum is of order $\mathcal{O} \left( \sqrt{N} (\log N)^{1/2+\ve} \right)$ for almost all $x$; there is a gap between this upper bound and the lower bound contained in Corollary~\ref{co1}, which remains open.

\section{The probabilistic model} \label{secmodel}

Let a real number $p \in (0,1)$ and a positive integer $\lambda$ be given. Let $(\Omega,\mathcal{A},\p)$ be a probability space, on which we can define a sequence $\xi_1, \xi_2, \dots$ of independent, identically distributed (i.i.d.) random variables, such that $\p(\xi_k=1)=p$ and $\p(\xi_k=0)=1-p$ for each $k \geq 1$. We are going to assign to each $k \geq 1$ a set $\mathcal{S}_k$ of $\lambda$ positive integers. Then the sequence $(m_k)_{k \geq 1}$ will be defined in such a way that it consists of all elements of $\mathcal{S}_k$ for those $k$ for which $\xi_k=1$, sorted in increasing order. Clearly a sequence $(m_k)_{k \geq 1}$ defined in this way is random, in the sense that it depends on an $\omega \in \Omega$. Finally, to obtain the sequence $(n_k)_{k \geq 1}$ satisfying the conclusion of Theorem~\ref{th1} and Theorem~\ref{th2}, we choose the parameters $p$ and $\lambda$ appropriately and take a sequence which consists of all odd numbers and all numbers of the form $2m_k$ for a ``typical'' realization of a random sequence $(m_k)_{k \geq 1}$.\\

More precisely, we define positive integers $\psi(r),~r \geq 1,$ satisfying the recursive relation
$$
\psi(1) + 2 \psi(2) + \dots + r (\psi(r)-1) ~<~ r^r ~\leq~ \psi(1) + 2 \psi(2) + \dots + r \psi(r).
$$
Then we also have
$$
\psi(1) + 2\psi(2) + \dots + (r-1) \psi(r-1) ~\geq~ (r-1)^{r-1} ~>~ \psi(1) + 2 \psi(2) + \dots + (r-1) (\psi(r-1)-1),
$$
and consequently
$$
r (\psi(r)-1) ~<~ r^r - (r-1)^{r-1} ~<~ r \psi(r) + r -1
$$
and
$$
\left| \psi(r) - \frac{r^r - (r-1)^{r-1}}{r} \right| < 1.
$$
This implies
$$
\psi(r) \sim r^{r-1} \qquad \textrm{as $r \to \infty$}.
$$

We set
$$
\Psi(r) = \psi(1) + 2\psi(2) +  \dots + r \psi(r).
$$
Then since by construction $\Psi(r)-r < r^r \leq \Psi(r)$, we have
$$
\Psi(r) \sim r^r \qquad \textrm{as $r \to \infty$.}
$$

Now let any $k \geq 1$ be given. Then there exists a number $r=r(k)$ such that $k \in (\Psi(r-1),\Psi(r)]$. We can write $k$ in the form
\begin{equation} \label{kform}
k = \Psi(r - 1) + \nu r + \rho
\end{equation}
for some numbers $\nu=\nu(k) \in \{0, 1, \dots, \psi(r)-1\}$ and $\rho = \rho(k) \in \{1, \dots, r\}$. For each $k$ we define a set $\mathcal{S}_k$, which consists of the numbers 
$$
\mathcal{S}_k = \{\lambda \Psi(r - 1) + \lambda \nu r + \rho + j r:~j=0, 1, \dots, \lambda-1\}.
$$
Thus for $k \in (\Psi(r-1),\Psi(r)]$, the elements of the set $\mathcal{S}_k$ form an arithmetic progression of $\lambda$ elements, with step size $r$.  More specifically, we have
\begin{eqnarray*}
\mathcal{S}_{\Psi(r-1)+1} & = & \big\{\lambda \Psi(r-1) + 1, \quad \lambda \Psi(r-1) + r + 1, ~\dots, ~\lambda \Psi(r-1) + (\lambda-1) r + 1 \big\},\\
\mathcal{S}_{\Psi(r-1)+2} & = & \big\{\lambda \Psi(r-1) + 2, \quad \lambda \Psi(r-1) + r + 2, ~\dots, ~\lambda \Psi(r-1) + (\lambda-1) r + 2 \big\},\\
& \vdots & \\
\mathcal{S}_{\Psi(r-1)+r} & = & \big\{\lambda \Psi(r-1) + r, \quad \lambda \Psi(r-1) + 2r, ~\dots, ~\lambda \Psi(r-1) + \underbrace{(\lambda-1)r + r}_{=\lambda r} \big\}.
\end{eqnarray*}
That means, the sets $\mathcal{S}_{\Psi(r-1)+1}, \dots, \mathcal{S}_{\Psi(r-1)+r}$ are interlaced in such a way that they form a partition of the set $\{\lambda\Psi(r-1)+1, \dots, \lambda \Psi(r-1)+\lambda r\}$. In a similar way, the sets $\mathcal{S}_{\Psi(r-1)+r+1},\dots, \mathcal{S}_{\Psi(r-1)+2r}$ form a partition of the set $\{\lambda\Psi(r-1)+\lambda r+1, \dots, \lambda \Psi(r-1)+2 \lambda r\}$, etc. Together, all the sets $(\mathcal{S}_k)_{\Psi(r-1) < k \leq \Psi(r)}$ form a partition of $\{\lambda \Psi(r-1) +1, \dots, \lambda \Psi(r)\}$.\\

Saying it in a oversimplified way, the sets $\mathcal{S}_k$ are constructed in such a way that the sum of the variances of the random variables $\xi_k \sum_{\ell \in \mathcal{S}_k} f(\ell x)$ over $k \in (\Psi(r-1),\Psi(r)]$ is roughly equal to $\lambda^2 \|f\|^2 (\Psi(r)-\Psi(r-1)) \V \xi_k$ whenever $x$ is an element of a set $A_r$ (which consists of all numbers which are close to a rational number with denominator $r$). The number of indices for which $x$ is contained in such a set $A_r$ will have positive relative frequency; consequently, the LIL will hold with right-hand side $\sqrt{\lambda^2 \|f\|^2 \V \xi_k}$ for the sums 
\begin{equation} \label{sumsk}
\sum_{k=1}^N \xi_k \sum_{\ell \in \mathcal{S}_k} f(\ell x).
\end{equation}
Now the random sequence $(m_k)_{k \geq 1}$ is defined such that it contains the elements of $\mathcal{S}_k$ if and only if $\xi_k=1$; then the LIL for~\eqref{sumsk} can be rewritten into an LIL for $\sum_{k=1}^N f(m_k x)$. Finally, we will choose the parameters $p$ and $\lambda$ in an appropriate way and define a sequence $(n_k)_{k \geq 1}$ which consists of all odd numbers, plus of all the numbers $(2 m_k)_{k \geq 1}$ for a ``typical'' realization of a random sequence $(m_k)_{k \geq 1}$. This sequence $(n_k)_{k \geq 1}$ will then have the desired properties.\\

The following Section~\ref{prel} contains several auxiliary results. In Section~\ref{aux1} we will prove several properties of sums involving $\xi_1,\xi_2, \dots$, which hold almost surely on $(\Omega,\mathcal{A},\p)$. Subsequently in Section~\ref{aux2} we will show how these results concerning sequences of random variables can be interpreted as results concerning random sequences $(m_k)_{k \geq 1}$ of integers, and what they imply for $(n_k)_{k \geq 1}$. Finally, in Section~\ref{secproofs} we will prove Theorem~\ref{th1} and~\ref{th2}.

\section{Preliminaries} \label{prel}

The following lemmas are classical results from number theory, or, more precisely, from the theory of uniform distribution modulo one and discrepancy theory. They can be found for example in~\cite{dts,knu}.

\begin{lemma}[Equidistribution theorem] \label{lemmaunif}
For any irrational real number $x$ and any interval $A \subset [0,1]$ we have
$$
\lim_{N \to \infty} \frac{1}{N} \sum_{k=1}^N \mathds{1}_A (\langle k x \rangle) = \mu (A).
$$
\end{lemma}

\begin{lemma} \label{lemmadiscr}
For any $\ve>0$ we have
\begin{equation} \label{discres}
D_N^* (\langle x \rangle, \langle 2 x \rangle, \dots, \langle N x \rangle) = \mathcal{O} \left(\frac{\log N(\log \log N)^{1+\ve}}{N} \right) \quad \textrm{as $N \to \infty$}
\end{equation}
for almost all real numbers $x$.
\end{lemma}

By changing the argument from $x$ to $2 x$, it is easily seen that the discrepancy estimate~\eqref{discres} implies that we also have
$$
D_N^* (\langle 2x \rangle, \langle 4 x \rangle, \dots, \langle 2 N x \rangle) = \mathcal{O} \left(\frac{\log N(\log \log N)^{1+\ve}}{N} \right) \quad \textrm{as $N \to \infty$}
$$
for almost all $x$; that is, the conclusion of Lemma~\ref{lemmadiscr} remains valid if we replace the sequence $(k)_{k \geq 1}$ of all positive integers by the sequence $(2k)_{k \geq 1}$ of even positive integers. Since $\N$ can be partitioned into even and odd numbers, a similar result must also hold for $(2k-1)_{k \geq 1}$. This is the statement of the following lemma.
\begin{lemma} \label{lemmadiscr2}
For any $\ve>0$ we have
\begin{equation*}
D_N^* (\langle x \rangle, \langle 3 x \rangle, \dots, \langle (2N-1) x \rangle) = \mathcal{O} \left(\frac{\log N(\log \log N)^{1+\ve}}{N} \right) \quad \textrm{as $N \to \infty$}
\end{equation*}
for almost all real numbers $x$.
\end{lemma}

\begin{lemma}[Koksma's inequality] \label{koksinequ}
Let $f$ be a function which has bounded variation on the unit interval. Let $y_1, \dots, y_N$ be points in $[0,1]$. Then
$$
\left| \frac{1}{N} \sum_{k=1}^N f(y_k) - \int_0^1 f(y)~dy \right| \leq \left( \textup{Var}_{[0,1]} f \right) D_N^*(y_1, \dots, y_N).
$$
\end{lemma}

For the proof of Lemma~\ref{lemma1} we will need the following simple fact. For a proof, see for example~\cite[Theorem 2]{salat}.
\begin{lemma} \label{lemmadensity}
Let $(c_k)_{k \geq 1}$ be a non-increasing sequence of non-negative real numbers. Let $\mathcal{N}$ denote a subset of $\N$ which has positive lower asymptotic density. Then 
$$
\sum_{k=1}^\infty c_k = \infty \qquad \textrm{implies that} \qquad \sum_{k \in \mathcal{N}} c_k = \infty.
$$
\end{lemma}

\section{Sequences of random variables} \label{aux1}

In this section we prove several auxiliary results concerning sums which involve the random variables $\xi_1,\xi_2,\dots$ defined in Section~\ref{secmodel}. We first state all results, and give proofs afterward.

\begin{lemma} \label{lemma1}
For any trigonometric polynomial without constant term $g(x)$ we have, $\p$-almost surely, that
$$
\limsup_{N \to \infty} \frac{\left| \sum_{k=1}^N \xi_k \sum_{\ell \in \mathcal{S}_k} g(\ell x) \right|}{\sqrt{N \log \log N}} = \sqrt{2 p(1-p)} ~ \lambda \|g\| \qquad \textrm{for almost all $x$}. 
$$
\end{lemma}

\begin{lemma} \label{lemma2}
For any function $h(x)$ satisfying~\eqref{f} we have, $\p$-almost surely, that
$$
\limsup_{N \to \infty} \frac{\left| \sum_{k=1}^N \xi_k \sum_{\ell \in \mathcal{S}_k} h(\ell x) \right|}{\sqrt{N \log \log N}} \leq \frac{\lambda}{\sqrt{2}} ~ \|h\| \qquad \textrm{for almost all $x$}. 
$$
\end{lemma}

The functions $\mathbf{I}$ appearing in the next lemma were defined in \eqref{ia}.
\begin{lemma} \label{lemma3}
For any fixed $L \geq 1$ we have, $\p$-almost surely, that
$$
\limsup_{N \to \infty} ~\max_{s = 0, \dots, 2^L-1} ~ \sup_{0 \leq a \leq 2^{-L}} \frac{\left| \sum_{k=1}^N \xi_k \sum_{\ell \in \mathcal{S}_k} \mathbf{I}_{[s2^{-L},s2^{-L}+a]} (\ell x) \right|}{\sqrt{N \log \log N}} \leq  5 \lambda^{3/2} \sqrt{2^{-L}}
$$
for almost all $x$.
\end{lemma}

\begin{proof}[Proof of Lemma~\ref{lemma1}:] 
Let $g(x)$ be a trigonometric polynomial without constant term. We may assume that $g \not\equiv 0$, since otherwise the Lemma is trivial. For simplicity of writing we assume that $g$ is an even function (in other words, that it consists only of cosine-terms); the proof in the general case is exactly the same. We write
$$
g(x) = \sum_{j=1}^d a_j \cos 2 \pi j x,
$$
and define
\begin{equation} \label{ci}
G_k(x) = \sum_{\ell \in \mathcal{S}_k} g (\ell x)
\end{equation}
and
$$
X_k = X_k(\omega,x) = \xi_k (\omega) G_k(x).
$$
Note that for every fixed $x$ the sequence $X_1, X_2, \dots$ is a sequence of independent random variables on $(\Omega,\mathcal{A},\p)$. We clearly have
$$
\E X_k = p G_k(x)
$$
and
$$
\V X_k = p (1-p) G_k(x)^2.
$$
For $n \geq 1$ we have
$$
\mu \left( x \in [0,1]:~|\sin \lambda \pi n x| \leq n^{-2} \right) = \mu \left( x \in [0,1]:~|\sin \pi x| \leq n^{-2} \right) \leq \frac{2}{\pi n^2},
$$
where $\mu$ denotes the Lebesgue measure. The term on the right-hand side of this inequality is summable in $n$. Consequently by the Borel--Cantelli lemma we have
$$
|\sin \lambda \pi n x| > n^{-2} \quad \textrm{eventually}, \qquad \textrm{for almost all $x \in [0,1]$.}
$$
In other words, there  exists a positive function $s(x)$ such that for almost all $x$ we have
\begin{equation*} \label{kpi}
|\sin \lambda \pi n x| > \frac{s(x)}{n^2} \qquad \textrm{for all $n \geq 1$}.
\end{equation*}
Thus for almost all $x$ we also have
\begin{equation} \label{sinlam}
\left( \frac{\sin \lambda \pi n x}{\sin \pi n x} \right)^2 \geq (\sin \lambda \pi n x)^2 \geq \frac{s(x)^2}{n^4}, \qquad \textrm{for all $n \geq 1$}.
\end{equation}
Using the classical formula for a sum of trigonometric functions having frequencies along an arithmetic progression, for $k \in (\Psi(r-1), \Psi(r)]$ for some $r$ we obtain the identity
$$
G_k(x) = \sum_{j=1}^d a_j \frac{\sin \lambda \pi r j x}{\sin \pi r j x} \cos \left(\left( 2 \lambda \Psi(r-1)+2 \lambda \nu r + 2 \rho + (\lambda-1) r \right) \pi j x\right),
$$
where $\nu=\nu(k)$ and $\rho=\rho(k)$ are as in~\eqref{kform}. Furthermore, also for $\Psi(r-1) < k \leq \Psi(r)$ for some $r$, we have
$$
G_k(x)^2 = \sum_{j=1}^d \frac{a_j^2}{2}  \left(\frac{\sin \lambda \pi r j x}{\sin \pi r j x}\right)^2 + C_k(x),
$$
where, writing $\cos_k[m]=\cos \left(\left( 2\lambda \Psi(r-1)+2 \lambda \nu r + 2 \rho + (\lambda-1) r \right) \pi m x\right)$, the function $C_k(x)$ is defined as
\begin{eqnarray*}
C_k(x) & = & \sum_{\substack{1 \leq j_1,j_2 \leq d,\\j_1 \neq j_2}} \frac{a_{j_1}a_{j_2}}{2} \left(\frac{\sin \lambda \pi r j_1 x}{\sin \pi r j_1 x}\right) \left(\frac{\sin \lambda \pi r j_2 x}{\sin \pi r j_2 x}\right) \cos_k[j_1-j_2] \\
& & \quad + \sum_{1 \leq j_1,j_2 \leq d} \frac{a_{j_1} a_{j_2}}{2} \left(\frac{\sin \lambda \pi r j_1 x}{\sin \pi r j_1 x}\right) \left(\frac{\sin \lambda \pi r j_2 x}{\sin \pi r j_2 x}\right)  \cos_k[j_1+j_2].
\end{eqnarray*}
Note that we have
\begin{equation} \label{sinfrac}
\left(\frac{\sin \lambda \pi r j x}{\sin \pi r j x}\right)^2 \leq \lambda^2, \qquad \textrm{for all $r \geq 1$ and $1 \leq j \leq d$.}
\end{equation}
Let $\ve>0$ be fixed. Using the orthogonality of the trigonometric system it is easy to show that for any fixed $j_1 \neq j_2$ the sum 
$$
\sum_{k=1}^{N} \cos \left(\left( 2 \lambda \Psi(r-1)+2 \lambda \nu r + 2 \rho + (\lambda-1) r \right) \pi (j_1 - j_2) x\right)
$$
is of order $\mathcal{O}(N^{1/2+\ve})$ for almost all $x$ as $N \to \infty$ (note that all frequencies in this sum are different, for different values of $k$). The same conclusion holds if $(j_1 - j_2)$ is replaced by $(j_1 + j_2)$, where the condition $j_1 \neq j_2$ is not necessary. Consequently, we also have that for almost all $x$
$$
\sum_{k=1}^{N} C_k(x) = \mathcal{O}(N^{1/2+\ve}) \qquad \textrm{as $N \to \infty$}.
$$
Consequently, by~\eqref{sinfrac}, for any (fixed) $\ve>0$ we have for almost all $x$ that
\begin{equation} \label{sumvxk}
\sum_{k=1}^N \V X_k \leq (1+\ve) \lambda^2 p (1-p) \|g\|^2 N
\end{equation}
for sufficiently large $N$. Furthermore, writing 
\begin{equation} \label{vr}
V_r = \sum_{k=\Psi(r-1)+1}^{\Psi(r)} \V X_k(x)^2, \qquad r \geq 1,
\end{equation}
for almost all $x$ we have
\begin{equation} \label{variancesum}
V_r = (\Psi(r)-\Psi(r-1)) p (1-p)  \sum_{j=1}^d \frac{a_j^2}{2}  \left(\frac{\sin \lambda \pi r j x}{\sin \pi r j x}\right)^2 + \mathcal{O}\left( (\Psi(r))^{1/2+\ve} \right)
\end{equation}
as $r \to \infty$. Remember that $\Psi(r) \sim r^r$. Thus  by~\eqref{sinlam} for almost all $x$ we have
\begin{equation} \label{varianceinbr}
V_r \geq (1-\ve) p(1-p) \|g\|^2 (\Psi(r)-\Psi(r-1)) \frac{s(x)^2}{d^4 r^4}
\end{equation}
for all sufficiently large $r$, which implies that for almost all $x$ 
\begin{equation} \label{varbound}
\sum_{k=1}^{\Psi(r)} \V (X_k) \to \infty \qquad \textrm{as} \qquad r \to \infty.
\end{equation}
Now~\eqref{varbound}, together with the fact that the random variables $X_k$ are uniformly bounded (we have $|X_k| \leq \lambda \|g\|_\infty$), implies that by Kolmogorov's law of the iterated logarithm for almost all $x$ we have
$$
\limsup_{N \to \infty} \frac{\left| \sum_{k=1}^N (X_k - \E X_k) \right|}{\sqrt{2 \left( \V(X_1) + \dots + \V(X_N) \right) \log \log \left( \V (X_1) + \dots + \V (X_N)\right)}} = 1 \qquad \textup{almost surely.}
$$
By~\eqref{sumvxk} we can conclude that for almost all $x$
\begin{equation} \label{upperb}
\limsup_{N \to \infty} \frac{\left| \sum_{k=1}^N  (X_k - \E X_k) \right|}{\sqrt{N \log \log N}} \leq \sqrt{2 (1+\ve) p (1-p)} \lambda \|g\| \qquad \textup{almost surely.}
\end{equation}
This establishes the upper bound in the LIL.\\

Next we want to apply the Berry--Esseen theorem to get the lower bound in the LIL in Lemma~\ref{lemma1}. We have
\begin{eqnarray*}
\E (|X_k - \E X_k |^3 ) & \leq & |G_k(x)|^3 ~\underbrace{\E (|\xi_k - \E \xi_k |^3)}_{\leq 1} \\
& \leq & \lambda \|g\|_\infty G_k(x)^2,
\end{eqnarray*}
which together with~\eqref{sinfrac} and~\eqref{variancesum} implies that for almost all $x$ we have
\begin{equation} \label{thirdmo}
\sum_{k = \Psi(r-1)+1}^{\Psi(r)} \E (|X_k - \E X_k |^3 ) = \mathcal{O} \left(  \Psi(r)-\Psi(r-1) \right) \quad \textrm{as $r \to \infty$.}
\end{equation}
By the Berry--Esseen theorem there exist an absolute constant $c_{\textup{abs}}$ such that for all $y \in \R$ we have
\begin{equation} \label{berryess}
\left| \p \left(\sum_{k=\Psi(r-1)+1}^{\Psi(r)} (X_k - \E X_k) \leq y \sqrt{V_r} \right) - \Phi(y) \right| \leq c_{\textup{abs}} \frac{\sum_{k=\Psi(r-1)+1}^{\Psi(r)} \E (|X_k - \E X_k |^3 )}{V_r^{3/2}},
\end{equation}
where $\Phi$ denotes the standard normal distribution and $V_r$ was defined in~\eqref{vr}. By~\eqref{varianceinbr} and~\eqref{thirdmo} the expression on the right-hand side of~\eqref{berryess} is for almost all $x$ of order
\begin{equation} \label{errorterm}
\mathcal{O} \left( \frac{r^6}{(\Psi(r)-\Psi(r-1))^{1/2}} \right) \quad = \qquad \mathcal{O} \left(r^{-r/2+6}\right) \qquad \textrm{as $r \to \infty$.} 
\end{equation}
For some fixed $\ve>0$ we set, for $r \geq 1$, 
$$
p_r = \p \left( \left| \sum_{k=\Psi(r-1)+1}^{\Psi(r)} (X_k - \E X_k)  \right| \geq \sqrt{2 (1-\ve) V_r \log r} \right)
$$
and
$$
q_r =\p \left( \left| \sum_{k=\Psi(r-1)+1}^{\Psi(r)} (X_k - \E X_k)  \right| \geq \sqrt{2 (1-\ve)^3  \|g\|^2 \lambda^2 \Psi(r) p (1-p) \log \log \Psi(r)} \right).
$$
Then by~\eqref{berryess} and~\eqref{errorterm} for almost all $x$ we have 
\begin{eqnarray} \label{prtilde}
p_r & = & \underbrace{2 - 2 \Phi\left(\sqrt{2 (1-\ve) \log r} \right)}_{=: \tilde{p}_r} + \mathcal{O} \left(r^{-r/2+6}\right) \qquad \textrm{as $r \to \infty$},
\end{eqnarray}
which by standard estimates for the tail probabilities of the normal distribution implies
\begin{equation} \label{sumardiv}
\sum_{r=1}^\infty \tilde{p}_r = \infty.
\end{equation}
Now we define sets $A \subset [0,1]$ and $A_r \subset [0,1],~r \geq 1$, in the following way: We set
$$
A = \left\{x \in (0,1):~ \sum_{j=1}^d \frac{a_j^2}{2} \left( \frac{\sin \lambda \pi j x}{\sin \pi j x} \right)^2 \geq \|g\|^2 \lambda^2 (1-\ve) \right\},
$$
and, for any $r \geq 1$,
$$
A_r = \left\{ x \in [0,1]:  \langle r x \rangle \in A \right\}.
$$
Note that $A$ is the union of finitely many intervals. Thus by Lemma~\ref{lemmaunif} we have
\begin{equation} \label{sumR}
\frac{1}{R} \sum_{r=1}^R \mathds{1} (x \in A_r) = \frac{1}{R} \sum_{r=1}^R \mathds{1}_A (\langle r x \rangle) \to \mu(A), \qquad \textrm{for almost all $x$, ~as $R \to \infty$,}
\end{equation}
where again $\mu$ denotes the Lebesgue measure. Consequently, for almost all $x$ the set $\{r \geq 1:~x \in A_r\}$ has positive asymptotic density. By~\eqref{variancesum} and using the facts that $\Psi(r-1) = o(\Psi(r))$ and $\Psi(r) \sim r^r$ as $r \to \infty$, for almost all $x$ we have
\begin{equation} \label{Vrlower}
V_r  \log r \geq \|g\|^2 \lambda^2 (1-\ve)^2 \Psi(r) p (1-p) \log \log \Psi(r) 
\end{equation}
whenever $x \in A_r$, ~for all sufficiently large $r$. Thus for almost all $x$ we have
\begin{eqnarray} \label{qrpr}
q_r \geq p_r \cdot \mathds{1}_{A_r} (x) 
\end{eqnarray}
for sufficiently large $r$. Note that $(\tilde{p}_r)_{r \geq 1}$ is a non-increasing sequence of non-negative real numbers. Thus~\eqref{sumardiv},~\eqref{sumR} and Lemma~\ref{lemmadensity} imply that for almost all~$x$ we have
$$
\sum_{r=1}^\infty \tilde{p}_r \cdot \mathds{1}_{A_r} (x) = \sum_{r \geq 1:~r \in A_r} \tilde{p}_r = \infty.
$$
By~\eqref{prtilde} this implies that for almost all $x$ we also have
$$
\sum_{r \geq 1:~r \in A_r} p_r = \infty,
$$
and consequently by~\eqref{qrpr}, also for almost all $x$, we have
$$
\sum_{r=1}^\infty q_r = \infty.
$$
The sets used for the definition of $q_r$ are obviously independent for different values of~$r$. Consequently for almost all $x$, by the second Borel--Cantelli lemma, $\p$-almost surely infinitely many events
$$
\left| \sum_{k=\Psi(r-1)+1}^{\Psi(r)} (X_k - \E X_k)  \right| \geq  \sqrt{2 (1-\ve)^3  \lambda^2 \|g\|^2 \Psi(r) p (1-p) \log \log \Psi(r)}
$$
occur. We clearly have
\begin{eqnarray*}
& & \limsup_{N \to \infty} \frac{\left| \sum_{k=1}^N (X_k - \E X_k) \right|}{\sqrt{2 N \log \log N}} \\
 & \geq & \limsup_{r \to \infty} \frac{\left| \sum_{k=1}^{\Psi(r)} (X_k - \E X_k) \right|}{\sqrt{2 \Psi(r) \log \log \Psi(r)}} \\
& \geq & \limsup_{r \to \infty} \frac{\left| \sum_{k=\Psi(r-1)+1}^{\Psi(r)} (X_k - \E X_k) \right|}{\sqrt{2 \Psi(r) \log \log \Psi(r)}} - \underbrace{\limsup_{r \to \infty} \frac{\left| \sum_{k=1}^{\Psi(r-1)} (X_k - \E X_k) \right|}{\sqrt{2 \Psi(r) \log \log \Psi(r)}}.}_{\textrm{$= 0$ a.s. for a.e. $x$ by~\eqref{upperb}, since $\Psi(r-1) = o (\Psi(r))$.}}
\end{eqnarray*}
Thus we have for almost all $x$
$$
\limsup_{N \to \infty} \frac{\left| \sum_{k=1}^N (X_k - \E X_k) \right|}{\sqrt{N \log \log N}} \geq \sqrt{2 (1-\ve)^3 p(1-p)} \lambda \|g\| \qquad \textup{$\p$-almost surely}
$$
Since $\ve>0$ was arbitrary, together with~\eqref{upperb} we have shown that for almost all $x$ we have
$$
\limsup_{N \to \infty} \frac{\left| \sum_{k=1}^N (X_k - \E X_k) \right|}{\sqrt{N \log \log N}} = \sqrt{2 p(1-p)} \lambda \|g\| \qquad \textup{$\p$-almost surely.}
$$
By Fubini's theorem this means that $\p$-almost surely we have
\begin{equation} \label{concl}
\limsup_{N \to \infty} \frac{\left| \sum_{k=1}^N (X_k - \E X_k) \right|}{\sqrt{N \log \log N}} = \sqrt{2 p(1-p)} \lambda \|g\| \qquad \textup{for almost all $x$.}
\end{equation}
This is almost the conclusion of the lemma, except that in~\eqref{concl} we have $X_k - \E X_k$ instead of $X_k$. Clearly $\E X_k = p \sum_{k \in \mathcal{S}_k} g(\ell x)$. By construction the sets ($\mathcal{S}_k)_{k \geq 1}$ are interlaced in such a way that we have the following fact: for given $N \geq 1$ and for $r$ such that $N \in (\Psi(r-1),\Psi(r)]$, we have
\begin{equation} \label{deltaapprox}
\# \left\{ \left( \bigcup_{k=1}^N \mathcal{S}_k \right) ~\Delta~ \{1, \dots, \lambda N\} \right\} \quad \leq \quad r \quad = \quad o \left(\log (N)\right),
\end{equation}
where $\Delta$ denotes the symmetric difference. Thus we have
\begin{eqnarray*}
\limsup_{N \to \infty} \frac{\left| \sum_{k=1}^N \E X_k \right|}{\sqrt{N \log \log N}} = \limsup_{N \to \infty} \frac{p \left| \sum_{k=1}^{\lambda N} g(k x)  \right|}{\sqrt{N \log \log N}} = 0 \qquad \textrm{for almost all $x$,}
\end{eqnarray*}
by Lemma~\ref{lemmadiscr} and Lemma~\ref{koksinequ}. Together with~\eqref{concl} this proves Lemma~\ref{lemma1}.
\end{proof}

\begin{proof}[Proof of Lemma~\ref{lemma2}:]
We start the proof of Lemma~\ref{lemma2} similar to that of Lemma~\ref{lemma1}, with the function $g$ replaced by $h$. For $k \geq 1$ we set
$$
G_k(x) = \sum_{\ell \in \mathcal{S}_k} h(\ell x)
$$
and
$$
X_k = X_k (\omega, x) = \xi_k (\omega) G_k(x). 
$$
Using H\"older's inequality we get
\begin{eqnarray*}
G_k(x)^2 \leq \lambda \sum_{\ell \in \mathcal{S}_k} h(\ell x)^2, \label{vxicase2}
\end{eqnarray*}
which implies
\begin{equation} \label{vrep}
\V X_k \leq \frac{\lambda}{4} \sum_{\ell \in \mathcal{S}_k} h(\ell x)^2.
\end{equation}
Using~\eqref{deltaapprox} we have
\begin{equation} \label{hrep}
\sum_{k=1}^{N} \sum_{\ell \in \mathcal{S}_k} h(\ell x)^2 = \sum_{k=1}^{\lambda N} h (k x)^2 + \mathcal{O} (\log N) \qquad \textrm{as $N \to \infty$.}
\end{equation}
Since by assumption $h$ has bounded variation on $[0,1]$, the same is true for $h^2$ (by the well-known fact that the product of two functions of bounded variation is also of bounded variation). By Lemma~\ref{lemmadiscr} and Lemma~\ref{koksinequ} we have
$$
\lim_{N \to \infty} \frac{\sum_{k=1}^N h(k x)^2}{N} = \int_0^1 h(x)^2~dx = \|h\|^2
$$
for almost all $x$. Together with~\eqref{vrep} and~\eqref{hrep} this implies
\begin{equation} \label{vxicase3}
\limsup_{N \to \infty} \frac{\sum_{k=1}^N \V X_k}{N} \leq \frac{\lambda^2}{4} \|h\|^2.
\end{equation}
for almost all $x$. In the sequel we write 
$$
B_N = \sum_{k=1}^N \V X_k, \qquad N \geq 1.
$$
If for a certain value of $x$ we have $\lim_{N \to \infty} B_N < \infty$, then for this $x$ we clearly have
$$
\limsup_{N \to \infty} \frac{\left| \sum_{k=1}^N (X_k - \E X_k) \right|}{\sqrt{N \log \log N}} = 0 \qquad \textup{a.s.}
$$
On the other hand, if for a certain $x$ we have $B_N \to \infty$, then by Kolmogorov's law of the iterated logarithm for uniformly bounded, independent random variables we have
$$
\limsup_{N \to \infty} \frac{\left| \sum_{k=1}^N (X_k - \E X_k) \right|}{\sqrt{ 2 B_N \log \log B_N}} = 1 \qquad \textup{a.s.}
$$
Together with~\eqref{vxicase3} and applying Fubini's theorem this implies that $\p$-almost surely we have
$$
\limsup_{N \to \infty} \frac{\left| \sum_{k=1}^N (X_k - \E X_k) \right|}{\sqrt{N \log \log N}} \leq \frac{\lambda}{\sqrt{2}} \|h\| \qquad \textup{for almost all $x$.}
$$
As in the proof of Lemma~\ref{lemma1} we can show that for almost all $x$.
$$
\limsup_{N \to \infty} \frac{\left| \sum_{k=1}^N \E X_k \right|}{\sqrt{N \log \log N}} = 0
$$
This proves Lemma~\ref{lemma2}.
\end{proof}

\begin{proof}[Proof of Lemma~\ref{lemma3}:] We use an argument similar to the one in~\cite[Lemma 4]{fuku1}. First we note that for any fixed $j \in \{0, \dots, 2^{L}-1\}$ we have
\begin{eqnarray}
& & \sup_{0 \leq a \leq 2^{-L}} \left| \sum_{k=1}^N (\xi_k - \E \xi_k) \sum_{\ell \in \mathcal{S}_k}\mathbf{I}_{[j2^{-L},j2^{-L}+a]} (\ell x) \right| \nonumber\\
& \leq & \sup_{0 \leq a \leq 2^{-L}} \left| \sum_{k=1}^N (\xi_k - \E \xi_k) \sum_{\ell \in \mathcal{S}_k} \mathds{1}_{[j2^{-L},j2^{-L}+a]} (\langle \ell x\rangle) \right| \label{term1}\\
& & \quad + \sup_{0 \leq a \leq 2^{-L}} \left|\lambda a \sum_{k=1}^N (\xi_k - \E \xi_k) \right| \label{term2}.
\end{eqnarray}

To estimate~\eqref{term1}, let $j \in \{0, \dots, 2^L-1\}$ be fixed. For simplicity of writing, we will assume that $j=0$; the proof in the other cases is exactly the same. Let $J \in \{1, \dots, \lambda\}$ be a fixed number, and write $s_k^{(J)}$ for the $J$-th element of $\mathcal{S}_k$. Assume that $N \geq 1$ is given. We set
$$
A_N = \max_{1 \leq n \leq N} ~\sup_{0 \leq a \leq 2^{-L}}~ \sum_{k=1}^n (\xi_{k} - \E \xi_{k}) \mathds{1}_{[0,a]} \left( \left\langle s_k^{(J)} x \right\rangle \right).
$$
Let $M$ denote the cardinality of the set of those numbers from $(\langle s_k^{(J)} x \rangle)_{1 \leq k \leq N}$, which are contained in the interval $[0,2^{-L}]$. We assume that $x$ is irrational. Then we can define indices $k_1 < \dots < k_M$ such that 
$$
\left\{s_{k_1}, \dots, s_{k_M} \right\} = \left\{k \leq N:~\left\langle s_k^{(J)} \right\rangle \in [0,a] \right\}
$$
and
$$
\left\langle s_{k_1}^{(J)} x \right\rangle < \dots < \left\langle s_{k_M}^{(J)}x \right\rangle.
$$
Furthermore, we define $\mathcal{T}_{m,n} = \{j \leq m:~k_j \leq n\}$ and $A_{m,n} = \sum_{k \in \mathcal{T}_{m,n}} (\xi_k-\E \xi_k)$. Then we have
$$
A_N = \max_{1 \leq n \leq N} ~\max_{1 \leq m \leq M}~ A_{m,n}.
$$
Let $y>1$ be a real number, to be determined later. We define random variables $\bar{n}$ and $\bar{m}$ by 
$$
\bar{n} = \min \left\{n:~\max_{1 \leq m \leq M} A_{m,n} > y \right\}, \qquad \bar{m} = \min \left\{m:~A_{m,\bar{n}}> y \right\}.
$$
Then, writing $C_{m,n}$ for the sets $\{\bar{n}=n,~\bar{m}=m\}$ we have a disjoint decomposition 
$$
\{A_N > y\} = \bigcup_{1 \leq n \leq N,~1 \leq m \leq M} C_{m,n}.
$$
The set $C_{m,n}$ belongs to the $\sigma$-field generated by $\xi_1, \dots, \xi_n$, and consequently it is independent of $A_{m,N}-A_{m,n}$, which only depends on $\xi_{n+1}, \dots, \xi_N$. It is known that the median of a random variable $X$ having binomial distribution $B(n,p)$ must always be one of numbers $\lfloor np \rfloor$ and $\lceil np \rceil$, which implies that $\p(X \geq np -1) \geq 1/2$ (see \cite{kaas}). Applying this to our situation we get $\p(A_{m,N} - A_{m,n} \geq -1) \geq 1/2$. Consequently we have
\begin{eqnarray*}
\p (C_{m,n}) & \leq & 2 \p \left(C_{m,n} \right) \p (A_{m,N}-A_{m,n} \geq -1) \\
& \leq & 2 \p \left(C_{m,n} \cap \{A_{m,N} > y-1\} \right) \\
& \leq & 2 \p \left( C_{m,n} \cap \left\{ \max_{r \leq M} A_{r,N} > y-1 \right\} \right).
\end{eqnarray*}
Summing over $m$ and $n$ we obtain
\begin{eqnarray}
\p (A_N > y) & \leq & 2 \p \left( A_N > y,~ \max_{r \leq M} A_{r,N} > y-1\right) \nonumber\\
& = & 2 \p \left( \max_{r \leq M} A_{r,N} > y-1 \right) \nonumber\\
& = & 2 \p \left(\max_{1 \leq r \leq M} \sum_{j=1}^r \mathds{1}_{[0,2^{-L}]} \left(\left\langle s_{k_j}^{(J)} \right\rangle \right)~(\xi_{k} - \E \xi_{k}) > y - 1 \right). \label{4p}
\end{eqnarray}
By the maximal version of Bernstein's inequality (see for example \cite[Lemma 2.2]{einmahl}), for any  independent, zero-mean random variables $Z_1, \dots, Z_M$ having variances $\sigma^2$ each and satisfying $|Z_k| \leq 1,~k \geq 1,$ we have
\begin{equation} \label{bernst}
\p \left(\max_{1 \leq r \leq M} \sum_{k=1}^r Z_k > t \right) \leq \exp \left(\frac{-t^2}{2 \sigma^2 M + 2t/3} \right), \qquad \textrm{for any $t > 0$.}
\end{equation}
By~\eqref{deltaapprox} we have $s_k^{(J)} \leq \lambda k + o (\log k)$ as $k \to \infty$. Consequently, by Lemma~\ref{lemmaunif}, for almost all $x$ for sufficiently large $N$ we have
$$
\sum_{k=1}^N \mathds{1}_{[0,2^{-L}]} \left( \left\langle s_k^{(J)} x \right\rangle \right) \leq \sum_{k=1}^{\lambda N+ \log N} \mathds{1}_{[0,2^{-L}]} (\langle k x \rangle) \leq 2^{-L+1} \lambda N.
$$
Combining this fact with~\eqref{bernst} we obtain the upper bound
$$
2 \exp \left(\frac{-(y-1)^2}{2^{-L+2} \lambda N+ 2y/3} \right)
$$
for the term in~\eqref{4p}. Using this estimate for $N={2^n}$ and $y = \sqrt{2 \lambda 2^n 2^{-L+2} \log \log 2^n}+1$, we get
$$
\sum_{n=1}^\infty \p \left( A_{2^n} > \sqrt{8 \lambda 2^n 2^{-L} \log \log 2^n} \right) < \infty.
$$
Consequently, by the Borel--Cantelli lemma, for almost all $x$ we have
$$
\limsup_{N \to \infty} \sup_{0 \leq a \leq 2^{-L}} \frac{\sum_{k=1}^N \mathds{1}_{[0,a]} \left(\left\langle s_k^{(J)} x \right\rangle \right)~(\xi_{k} - \E \xi_{k})}{\sqrt{N \log \log N}} \leq 4 \sqrt{\lambda 2^{-L}} \qquad \textup{almost surely}.
$$
Repeating the same argument with negative signs, we obtain
$$
\limsup_{N \to \infty} \sup_{0 \leq a \leq 2^{-L}} \frac{\left| \sum_{k=1}^N \mathds{1}_{[0,a]} \left(\left\langle s_k^{(J)} x \right\rangle \right)~(\xi_{k} - \E \xi_{k}) \right|}{\sqrt{N \log \log N}} \leq 4 \sqrt{\lambda 2^{-L}} \qquad \textup{almost surely}.
$$
This result holds independent of the choice of $J$. Since
$$
\sum_{k=1}^N (\xi_k - \E \xi_k) \sum_{\ell \in \mathcal{S}_k} \mathds{1}_{[0,a]} \left( \left\langle \ell x \right\rangle \right) = \sum_{J=1}^\lambda \sum_{k=1}^N (\xi_k - \E \xi_k) \mathds{1}_{[0,a]} \left( \left\langle s_k^{(J)} \right\rangle \right),
$$
this implies that for almost all $x$
$$
\limsup_{N \to \infty} \sup_{0 \leq a \leq 2^{-L}} \frac{\left| \sum_{k=1}^N (\xi_{k} - \E \xi_{k}) \sum_{\ell \in \mathcal{S}_k} \mathds{1}_{[0,a]} \left(\left\langle \ell x \right\rangle \right) \right|}{\sqrt{N \log \log N}} \leq 4 \lambda^{3/2} \sqrt{2^{-L}} \qquad \textup{almost surely}.
$$
For the term in \eqref{term2}, by the LIL for i.i.d. random variables we have
\begin{eqnarray*}
& & \limsup_{N \to \infty}  \sup_{0 \leq a \leq 2^{-L}} \frac{\left|\lambda a \sum_{k=1}^N (\xi_k - \E \xi_k) \right|}{\sqrt{N \log \log N}} \\
& \leq &  \lambda 2^{-L} \limsup_{N \to \infty}  \frac{\left|\sum_{k=1}^N (\xi_k - \E \xi_k) \right|}{\sqrt{N \log \log N}} \quad \leq \quad \frac{\lambda 2^{-L}}{\sqrt{2}} \qquad \textup{almost surely},
\end{eqnarray*}
where we used $\V \xi_k \leq 1/4$. This proves that for almost all $x$ we have
\begin{equation*} 
\limsup_{N \to \infty} ~\sup_{0 \leq a \leq 2^{-L}} \frac{\left| \sum_{k=1}^N (\xi_{k} - \E \xi_{k}) \sum_{\ell \in \mathcal{S}_k} \mathbf{I}_{[0,a]} \left(\left\langle \ell x \right\rangle \right) \right|}{\sqrt{N \log \log N}} \leq 5 \lambda^{3/2} \sqrt{2^{-L}} \qquad \textup{almost surely}.
\end{equation*}
As in the proof of Lemma~\ref{lemma1} and Lemma~\ref{lemma2} we can use Lemma~\ref{lemmadiscr} to show that for almost all $x$ we have
$$
\limsup_{N \to \infty} ~\sup_{0 \leq a \leq 2^{-L}} \frac{\left| \sum_{k=1}^N \E \xi_{k} \sum_{\ell \in \mathcal{S}_k} \mathbf{I}_{[0,a]} \left(\left\langle \ell x \right\rangle \right) \right|}{\sqrt{N \log \log N}} = 0.
$$
Thus for almost all $x$ we have
\begin{equation} \label{result}
\limsup_{N \to \infty} ~\sup_{0 \leq a \leq 2^{-L}} \frac{\left| \sum_{k=1}^N \xi_{k} \sum_{\ell \in \mathcal{S}_k} \mathbf{I}_{[0,a]} \left(\left\langle \ell x \right\rangle \right) \right|}{\sqrt{N \log \log N}} \leq 5 \lambda^{3/2} \sqrt{2^{-L}} \qquad \textup{almost surely}.
\end{equation}
Again by Fubini's theorem we can conclude that $\p$-almost surely the asymptotic result~\eqref{result} holds for almost all $x$. The same result holds with the intervals $[0,a]$ replaced by $[j 2^{-L},j2^{-L}+a]$ for some $j \in \{1, \dots, 2^L-1\}$. This proves the lemma.
\end{proof}

\section{Random sequences} \label{aux2}

As already mentioned in Section~\ref{secmodel}, we can use the random variables $\xi_1,\xi_2,\dots$ and the sets $\mathcal{S}_k,~k \geq 1,$ to define a (random) sequence $(m_k)_{k \geq 1}$ of positive integers in the following way: for $\omega \in \Omega$ we require that the sequence $(m_k)_{k \geq 1}=(m_k(\omega))_{k \geq 1}$ consists of all the numbers which are contained in the sets $\mathcal{S}_k$ for which $\xi_k=1$, sorted in increasing order.\\

Note that a typical realization of a sequence $(m_k)_{k \geq 1}$ does \emph{not} satisfy the assumptions of Theorem~\ref{th1} and~\ref{th2}, since by the Erd\H os--R\'enyi  ``pure heads'' theorem with probability one such a sequence will have gaps for $m_{k+1} - m_k$ of order roughly $\log k$, infinitely often (see for example~\cite{gsw, turi}). Thus we define a second sequence $(n_k)_{k \geq 1}=(n_k(\omega))_{k \geq 1}$ which for a given $\omega$ and corresponding random sequence $(m_k)_{k \geq 1}$ contains all the number
$$
2k-1, ~k \geq 1, \qquad \textrm{and} \qquad 2 m_k,~k \geq 1,
$$
sorted in increasing order. Thus independent of $\omega$ the sequence $(n_k)_{k \geq 1}$ always contains all odd numbers, which implies that $n_{k+1} - n_k \leq 2$ for all $\omega$.\\

For any $N \geq 1$ we define a random variable $K(N)$ by
$$
K(N) = \sum_{k=1}^N \xi_k.
$$
Then by the strong law of large numbers we have
\begin{equation} \label{approx1}
\lim_{N \to \infty} \frac{K(N)}{N} = p, \qquad \textrm{$\p$-almost surely.}
\end{equation}
Furthermore, by~\eqref{deltaapprox} we have
\begin{equation} \label{approx2}
\# \left\{ \{k:~m_k \leq N\} ~\Delta~ \left\{1, \dots, K(\lfloor N/\lambda \rfloor) \right\} \right\} = o (\log N)
\end{equation}
and
\begin{equation} \label{approx3}
\# \left\{ \{m_k:~m_k \leq N\} ~\Delta~ \left\{ \bigcup_{k \leq N/\lambda:~\xi_k = 1} \mathcal{S}_k \right\} \right\} = o (\log N)
\end{equation}
as $N \to \infty$. By~\eqref{approx3} for any function $f$ satisfying~\eqref{f} we have
\begin{equation} \label{flil}
\limsup_{N \to \infty} \frac{\left| \sum_{k:~m_k \leq N} f(m_k x) \right|}{\sqrt{N \log \log N}} = \limsup_{N \to \infty} \frac{\left| \sum_{k \leq N/\lambda} \xi_k \sum_{\ell \in \mathcal{S}_k} f(\ell x) \right|}{\sqrt{N \log \log N}}.
\end{equation}
For any given $N \geq 1$ we have
\begin{equation} \label{sets}
\{n_k:~n_k \leq N\} = \{2k-1:~ 2k-1 \leq N \} \cup \{2m_k:~2m_k \leq N\}.
\end{equation}
Thus for any trigonometric polynomial $g$ without constant term we have, $\p$-almost surely, that
\begin{eqnarray}
\limsup_{N \to \infty} \frac{\left| \sum_{k:~n_k \leq N} g(n_k x) \right|}{\sqrt{N \log \log N}} & = & \limsup_{N \to \infty} \frac{\left| \sum_{k=1}^{\lceil N/2 \rceil} g((2k-1) x) + \sum_{k:~2 m_k \leq N} g(2m_k x) \right|}{\sqrt{N \log \log N}} \nonumber \\
& = & \limsup_{N \to \infty} \frac{\left| \sum_{k:~2 m_k \leq N} g(2m_k x) \right|}{\sqrt{N \log \log N}} \label{flilest}\\
& = & \sqrt{p (1-p)} \sqrt{\lambda} \|g\| \qquad \textup{for almost all $x$,} \label{lilnk}
\end{eqnarray}
where we used~\eqref{flil} and Lemma~\ref{lemma1} to calculate~\eqref{flilest}, and Lemma~\ref{lemmadiscr2} and Lemma~\ref{koksinequ} to show that
$$
\limsup_{N \to \infty} \frac{\left| \sum_{k=1}^{\lceil N/2 \rceil} g((2k-1) x)\right|}{\sqrt{N \log \log N}} = 0 \qquad \textup{for almost all $x$.}
$$
By~\eqref{approx1},~\eqref{approx2} and~\eqref{sets} we have, $\p$-almost surely, that
$$
\# \{k:~n_k \leq N\} \sim \frac{N}{2} + K(\lfloor N/(2 \lambda)\rfloor) \sim N \underbrace{\left( \frac{1}{2} + \frac{p}{2\lambda} \right)}_{=(\lambda+p)/(2\lambda)} \qquad \textrm{as $N \to \infty$.}
$$
Consequently by~\eqref{lilnk} we have, $\p$-almost surely, that
\begin{eqnarray} 
\limsup_{N \to \infty} \frac{\left| \sum_{k=1}^N g(n_k x) \right|}{\sqrt{N \log \log N}} & = & \sqrt{p (1-p)} \sqrt{\lambda} \frac{\sqrt{2\lambda}}{\sqrt{\lambda+p}} \|g\| \nonumber\\
& = & \frac{\lambda \sqrt{2 p (1-p)}}{\sqrt{\lambda+p}} \|g\| \qquad  \textup{for almost all $x$.} \label{lilf1}
\end{eqnarray}

In a similar way we can modify Lemma~\ref{lemma2} and Lemma~\ref{lemma3}, and reformulate them in terms of $(n_k)_{k \geq 1}$. Instead of Lemma~\ref{lemma2} we get the following: for any function $h(x)$ satisfying~\eqref{f} we have, $\p$-almost surely, that
\begin{equation} \label{lilf2}
\limsup_{N \to \infty} \frac{\left| \sum_{k=1}^N h(n_k x) \right|}{\sqrt{N \log \log N}} \leq \frac{\lambda}{\sqrt{2}\sqrt{\lambda+p}} ~ \|h\| \qquad \textrm{for almost all $x$}. 
\end{equation}
Instead of Lemma~\ref{lemma3} we get the following: for any fixed $L \geq 1$ we have, $\p$-almost surely, that
\begin{equation} \label{lilf3}
\limsup_{N \to \infty} ~\max_{s = 0, \dots, 2^L-1} ~ \sup_{0 \leq a \leq 2^{-L}} \frac{\left| \sum_{k=1}^N \mathbf{I}_{[s2^{-L},s2^{-L}+a]} (n_k x) \right|}{\sqrt{N \log \log N}} \leq  \frac{5 \lambda^{3/2}}{\sqrt{\lambda+p}} \sqrt{2^{-L}}
\end{equation}
for almost all $x$.\\

Note that any function $f$ satisfying~\eqref{f} can be split into a sum $g+h$ of a trigonometric polynomial $g$ (without constant term) and a remainder function $h$, where $\|h\|$ can be made arbitrarily small. Combining~\eqref{lilf1} and~\eqref{lilf2} and letting $\|h\| \to 0$ we obtain the following lemma.

\begin{lemma} \label{lemma1b}
For any function $f(x)$ satisfying~\eqref{f} we have, $\p$-almost surely, that
\begin{equation*}
\limsup_{N \to \infty} \frac{\left| \sum_{k=1}^N f(n_k x) \right|}{\sqrt{N \log \log N}} = \frac{\lambda \sqrt{2 p (1-p)}}{\sqrt{\lambda+p}} \|f\|\qquad \textrm{for almost all $x$}. 
\end{equation*}
\end{lemma}

For the following calculations, we have to introduce the modified discrepancies $D_N^{(\geq 2^{-L})}$ and $D_N^{(\leq 2^{-L})}$, which only consider ``large'' and ``small'' intervals, respectively. More precisely, for any integer $L \geq 1$ and points $y_1, \dots, y_N$, we set
$$
D_N^{(\geq 2^{-L})} (y_1, \dots, y_N) = \max_{s = 0, \dots, 2^L-1}~ \sup_{0 \leq a \leq 2^{-L}} \left| \sum_{k=1}^N \mathbf{I}_{[s2^{-L},s2^{-L}+a]} (y_k) \right|
$$
and
$$
D_N^{(\leq 2^{-L})} (y_1, \dots, y_N) = \max_{0 < s < 2^L}  \left| \sum_{k=1}^N \mathbf{I}_{[0, s 2^{-L}]} (y_k) \right|.
$$
Note that any subinterval of $[0,1]$ which has one vertex at the origin can be written as the disjoint union of (at most) one interval of the form $[0, s 2^{-L}]$ for some appropriate $s$ and (at most) one interval of the form $[s2^{-L},s2^{-L}+a]$ for appropriate $s$ and $a$; consequently, for any points $y_1, \dots, y_N$ we always have
\begin{equation} \label{discrepancies}
D_N^{(\geq 2^{-L})} (y_1, \dots, y_N) \leq D_N^* (y_1, \dots, y_N) \leq D_N^{(\geq 2^{-L})} (y_1, \dots, y_N)  + D_N^{(\leq 2^{-L})} (y_1, \dots, y_N).
\end{equation}
By Lemma~\ref{lemma1b}, and since for the discrepancy $D_N^{(\geq 2^{-L})}$ only finitely many indicator functions are considered, we have, $\p$-almost surely, that
\begin{eqnarray}
\limsup_{N \to \infty} \frac{N D_N^{(\geq 2^{-L})} (\langle n_1 x \rangle, \dots, \langle n_N x \rangle)}{\sqrt{N \log \log N}} & = & \frac{\lambda \sqrt{2 p (1-p)}}{\sqrt{\lambda+p}} \underbrace{\max_{0 \leq s < 2^L} \|\mathbf{I}_{[0, s 2^{-L}]}\|}_{= 1/2} \nonumber\\
& = & \frac{\lambda \sqrt{p (1-p)}}{\sqrt{2} \sqrt{\lambda+p}} \qquad \textrm{for almost all $x$.} \label{discrl2}
\end{eqnarray}
Clearly $L$ can be made arbitrarily large. Thus by~\eqref{lilf3},~\eqref{discrepancies} and~\eqref{discrl2} we obtain the following lemma.

\begin{lemma} \label{lemma3b}
For $\p$-almost all sequences $(n_k)_{k \geq 1}$ we have
$$
\limsup_{N \to \infty} \frac{N D_N^*(\langle n_1 x \rangle, \dots, \langle n_N x \rangle)}{\sqrt{N \log \log N}}  = \frac{\lambda \sqrt{p (1-p)}}{\sqrt{2} \sqrt{\lambda+p}} \qquad \textrm{for almost all $x$.}
$$
\end{lemma}

\section{Proof of the theorems} \label{secproofs}

Theorem~\ref{th1} follows from Lemma~\ref{lemma1b}. In fact, assume that a real number $\Lambda \geq 0$ is given. If $\Lambda=0$, then we may choose $n_k=k,~k \geq 1.$ By Lemma~\ref{lemmadiscr} this sequence satisfies the conclusion of Theorem~\ref{th1}. If $\Lambda>0$, then we choose a positive integer $\lambda$ and a real number $p \in (0,1)$ such that
$$
\Lambda = \frac{\lambda \sqrt{2 p (1-p)}}{\sqrt{\lambda+p}}.
$$
For these values of $\lambda$ and $p$, we can use the probabilistic construction from Section~\ref{secmodel} to construct a class of random sequences $(n_k)_{k \geq 1}$ as described in the previous sections, each of them satisfying the growth condition $n_{k+1}-n_k \in \{1,2\}$. By Lemma~\ref{lemma1b} for $\p$-almost all such sequences $(n_k)_{k \geq 1}$ we have, for any fixed function $f$ satisfying~\eqref{f}, that
\begin{equation} \label{lilf4}
\limsup_{N \to \infty} \frac{\left| \sum_{k=1}^N f(n_k x) \right|}{\sqrt{N \log \log N}} = \Lambda \|f\|\qquad \textrm{for almost all $x$}.
\end{equation}
This proves Theorem~\ref{th1}. Theorem~\ref{th2} can be deduced from Lemma~\ref{lemma3b} in a similar way.

\def\cprime{$'$}


\begin{thebibliography}{10}

\bibitem{ab}
C.~Aistleitner and I.~Berkes.
\newblock Probability and metric discrepancy theory.
\newblock {\em Stoch. Dyn.}, 11(1):183--207, 2011.

\bibitem{abs}
C.~Aistleitner, I.~Berkes, and K.~Seip.
\newblock {G}{C}{D} sums from {P}oisson integrals and systems of dilated
  functions.
\newblock {\em J. Europ. Math. Soc.}
\newblock To appear. Available at \url{http://arxiv.org/abs/1210.0741}.

\bibitem{arnold}
V.~I. Arnol{\cprime}d.
\newblock To what extent are arithmetic progressions of fractional parts
  random?
\newblock {\em Uspekhi Mat. Nauk}, 63(2(380)):5--20, 2008.

\bibitem{baker}
R.~C. Baker.
\newblock Metric number theory and the large sieve.
\newblock {\em J. London Math. Soc. (2)}, 24(1):34--40, 1981.

\bibitem{berkes}
I.~Berkes.
\newblock A central limit theorem for trigonometric series with small gaps.
\newblock {\em Z. Wahrsch. Verw. Gebiete}, 47(2):157--161, 1979.

\bibitem{berkes2}
I.~Berkes.
\newblock Probability theory of the trigonometric system.
\newblock In {\em Limit theorems in probability and statistics ({P}\'ecs,
  1989)}, volume~57 of {\em Colloq. Math. Soc. J\'anos Bolyai}, pages 35--58.
  North-Holland, Amsterdam, 1990.

\bibitem{beph}
I.~Berkes and W.~Philipp.
\newblock The size of trigonometric and {W}alsh series and uniform distribution
  {${\rm mod}\ 1$}.
\newblock {\em J. London Math. Soc. (2)}, 50(3):454--464, 1994.

\bibitem{bobkov}
S.~G. Bobkov and F.~G{\"o}tze.
\newblock Concentration inequalities and limit theorems for randomized sums.
\newblock {\em Probab. Theory Related Fields}, 137(1-2):49--81, 2007.

\bibitem{borwein}
P.~Borwein and R.~Lockhart.
\newblock The expected {$L_p$} norm of random polynomials.
\newblock {\em Proc. Amer. Math. Soc.}, 129(5):1463--1472, 2001.

\bibitem{dts}
M.~Drmota and R.~F. Tichy.
\newblock {\em Sequences, discrepancies and applications}, volume 1651.
\newblock Springer-Verlag, Berlin, 1997.

\bibitem{einmahl}
U.~Einmahl and D.~M. Mason.
\newblock Some universal results on the behavior of increments of partial sums.
\newblock {\em Ann. Probab.}, 24(3):1388--1407, 1996.

\bibitem{erdelyi}
T.~Erd{\'e}lyi.
\newblock Polynomials with {L}ittlewood-type coefficient constraints.
\newblock In {\em Approximation theory, {X} ({S}t. {L}ouis, {MO}, 2001)},
  Innov. Appl. Math., pages 153--196. Vanderbilt Univ. Press, Nashville, TN,
  2002.

\bibitem{fuku1}
K.~Fukuyama.
\newblock A law of the iterated logarithm for discrepancies: non-constant
  limsup.
\newblock {\em Monatsh. Math.}, 160(2):143--149, 2010.

\bibitem{fuku2}
K.~Fukuyama.
\newblock Pure {G}aussian limit distributions of trigonometric series with
  bounded gaps.
\newblock {\em Acta Math. Hungar.}, 129(4):303--313, 2010.

\bibitem{fuku3}
K.~Fukuyama.
\newblock A central limit theorem for trigonometric series with bounded gaps.
\newblock {\em Probab. Theory Related Fields}, 149(1-2):139--148, 2011.

\bibitem{gapo}
V.~F. Gapo{\v{s}}kin.
\newblock Lacunary series and independent functions.
\newblock {\em Uspehi Mat. Nauk}, 21(6 (132)):3--82, 1966.

\bibitem{gsw}
L.~Gordon, M.~F. Schilling, and M.~S. Waterman.
\newblock An extreme value theory for long head runs.
\newblock {\em Probab. Theory Relat. Fields}, 72(2):279--287, 1986.

\bibitem{kaas}
R.~Kaas and J.~M. Buhrman.
\newblock Mean, median and mode in binomial distributions.
\newblock {\em Statist. Neerlandica}, 34(1):13--18, 1980.

\bibitem{kac}
M.~Kac.
\newblock Probability methods in some problems of analysis and number theory.
\newblock {\em Bull. Amer. Math. Soc.}, 55:641--665, 1949.

\bibitem{kahane}
J.-P. Kahane.
\newblock Lacunary {T}aylor and {F}ourier series.
\newblock {\em Bull. Amer. Math. Soc.}, 70:199--213, 1964.

\bibitem{knu}
L.~Kuipers and H.~Niederreiter.
\newblock {\em Uniform distribution of sequences}.
\newblock Wiley-Interscience [John Wiley \& Sons], New York, 1974.

\bibitem{plt}
W.~Philipp.
\newblock Limit theorems for lacunary series and uniform distribution {${\rm
  mod}\ 1$}.
\newblock {\em Acta Arith.}, 26(3):241--251, 1974/75.

\bibitem{salat}
T.~{\v{S}}al{\'a}t.
\newblock On subseries.
\newblock {\em Math. Z.}, 85:209--225, 1964.

\bibitem{salemzyg}
R.~Salem and A.~Zygmund.
\newblock Some properties of trigonometric series whose terms have random
  signs.
\newblock {\em Acta Math.}, 91:245--301, 1954.

\bibitem{schatte}
P.~Schatte.
\newblock On a law of the iterated logarithm for sums {${\rm mod}\,1$} with
  application to {B}enford's law.
\newblock {\em Probab. Theory Related Fields}, 77(2):167--178, 1988.

\bibitem{turi}
J.~T{\'u}ri.
\newblock Limit theorems for the longest run.
\newblock {\em Ann. Math. Inform.}, 36:133--141, 2009.

\bibitem{weber}
M.~Weber.
\newblock Discrepancy of randomly sampled sequences of reals.
\newblock {\em Math. Nachr.}, 271:105--110, 2004.

\end{thebibliography}
\end{document}